\documentclass[a4paper, abstract, 12pt]{scrartcl}
\usepackage{amsmath, amsthm, ascmac, amssymb, authblk} 
\usepackage[T1]{fontenc}
\usepackage{lmodern,exscale}
\DeclareFontFamily{U}{mathx}{\hyphenchar\font45}
\DeclareFontShape{U}{mathx}{m}{n}{
  <5> <6> <7> <8> <9> <10> gen * mathx
  <10.95> mathx10 <12> <14.4> <17.28> <20.74> <24.88> mathx12
  }{}
\DeclareSymbolFont{mathxSymbols}{U}{mathx}{m}{n}
\DeclareMathSymbol{\intop}{1}{mathxSymbols}{"B3} 
\DeclareMathSymbol{\ointop}{1}{mathxSymbols}{"B6} 
\begin{document}

\title{\LARGE{Instability of solitary waves
for a generalized derivative nonlinear Schr\"odinger equation \\
in a borderline case}}
\author{Noriyoshi Fukaya\thanks
{\textit{E-mail address:}
\texttt{noriyoshi\_fukaya@alumni.tus.ac.jp}}}
\date{}
\dedication{\normalsize{Department of Mathematics,
Graduate School of Science,
Tokyo University of Science,
1-3 Kagurazaka,
Shinjuku-ku,
Tokyo 162-8601,
Japan}}
\maketitle
\begin{abstract}
We study the orbital instability of solitary waves for a generalized derivative nonlinear Schr\"odinger equation.
We give sufficient conditions for instability of a two-parameter family of solitary waves in a degenerate case.
\end{abstract}

\theoremstyle{definition}
\newtheorem{definition}{Definition}
\newtheorem{remark}{Remark}
\theoremstyle{plain}
\newtheorem{theorem}{Theorem}
\newtheorem{proposition}{Proposition}
\newtheorem{lemma}{Lemma}
\let\mb\mathbb
\let\ve\varepsilon
\let\vp\varphi
\let\ol\overline
\let\pt\partial

\renewcommand{\labelenumi}{\textup{(\roman{enumi})}}

\section{Introduction} \label{sec:in}

In this paper, we consider the following generalized derivative nonlinear Schr\"odinger equation
\begin{equation} \label{gDNLS}
i\pt_tu
+\pt_x^2u
+i|u|^{2\sigma}\pt_xu
=0,\quad
(t,x)\in\mb R\times\mb R,
\end{equation}
where $u$ is a complex-valued function of 
$(t,x)\in\mb R\times\mb R$ and $\sigma\ge1$.
Eq.\ \eqref{gDNLS} appears in plasma physics, 
nonlinear optics, and so on (see, e.g., \cite{MOMT76, Mjolhus76}).

It is known that \eqref{gDNLS} has a two-parameter family of solitary waves
\begin{equation} \label{sw}
u_\omega(t,x)
=e^{i\omega_0t}\phi_\omega(x-\omega_1t),
\end{equation}
where 
$\omega=(\omega_0,\omega_1)\in\Omega
:=\{\,(\omega_0,\omega_1)\in\mb R^2
\mid\omega_1^2<4\omega_0\,\}$,
\begin{align*}
\phi_\omega(x)
&=\vp_\omega(x)\exp i\left(
\frac{\omega_1}2x
-\frac1{2\sigma+2}\int_{-\infty}^x\vp_\omega(y)^{2\sigma}\,dy
\right), \\
\vp_\omega(x)
&=\left\{
\frac{(\sigma+1)(4\omega_0-\omega_1^2)}
{2\sqrt{\omega_0}\cosh(\sigma\sqrt{4\omega_0-\omega_1^2}\,x)-\omega_1}
\right\}^{1/2\sigma}.
\end{align*}
We note that $\phi_\omega$ is a solution of 
\begin{equation} \label{SP}
-\pt_x^2\phi
+\omega_0\phi
+\omega_1i\pt_x\phi
-i|\phi|^{2\sigma}\pt_x\phi
=0,\quad x\in\mb R.
\end{equation}

We regard $L^2(\mb R)=L^2(\mb R,\mb C)$ and $H^1(\mb R)=H^1(\mb R,\mb C)$
as real Hilbert spaces with inner products
\begin{align*}
&(v,w)_{L^2}
=\Re\int_\mb R v(x)\ol{w(x)}\,dx,\quad
v,w\in L^2(\mb R), \\
&(v,w)_{H^1}
=(v,w)_{L^2}+(\pt_xv,\pt_xw)_{L^2},\quad
v,w\in H^1(\mb R).
\end{align*}

Recently, Hayashi and Ozawa \cite{HOp} proved that the Cauchy problem for \eqref{gDNLS} is locally well-posed in the energy space $H^1(\mb R)$ for all $\sigma\ge1$ (see also \cite{HO92, Hayashi93, HO94}).
Moreover, \eqref{gDNLS} has three conserved quantities
\begin{align*}
&E(u)
=\frac12\|\pt_xu\|_{L^2}^2
-\frac1{2(\sigma+1)}(i\pt_xu,|u|^{2\sigma}u)_{L^2}, \\
&Q_0(u)
=\frac12\|u\|_{L^2}^2,\quad
Q_1(u)
=\frac12(i\pt_xu,u)_{L^2}.
\end{align*}
Note that \eqref{gDNLS} can be written in Hamiltonian form 
$i\pt_tu(t)=E'(u(t))$.

For $\omega\in\Omega$ and $u\in H^1(\mb R)$, we define 
\[
S_\omega(u)
=E(u)+\sum_{j=0}^1\omega_jQ_j(u).
\]
Then \eqref{SP} is equivalent to $S_\omega'(\phi)=0$.
For $\omega\in\Omega$, let
$d(\omega)
=S_\omega(\phi_\omega)$.
Then,
\begin{align} \notag
d'(\omega)
&=(\pt_{\omega_0}d(\omega),\pt_{\omega_1}d(\omega))
=(Q_0(\phi_\omega),Q_1(\phi_\omega)), \\
d''(\omega)
&=\begin{bmatrix}
\pt_{\omega_0}^2d(\omega) & \pt_{\omega_0}\pt_{\omega_1}d(\omega) \\
\pt_{\omega_1}\pt_{\omega_0}d(\omega) & \pt_{\omega_1}^2d(\omega)
\end{bmatrix}
=\begin{bmatrix} \label{Hd}
\langle Q'_0(\phi_\omega),\pt_{\omega_0}\phi_\omega\rangle &
\langle Q'_1(\phi_\omega),\pt_{\omega_0}\phi_\omega\rangle \\
\langle Q'_0(\phi_\omega),\pt_{\omega_1}\phi_\omega\rangle &
\langle Q'_1(\phi_\omega),\pt_{\omega_1}\phi_\omega\rangle
\end{bmatrix}.
\end{align}

The stability of solitary waves is defined as follows.
%%%%%%%%%%%%%%%
\begin{definition}
The solitary wave $e^{i\omega_0t}\phi_\omega(\cdot-\omega_1t)$ is said to be \textit{stable} if for each $\ve>0$ there exists $\delta>0$ with the following property.
If $u_0\in H^1(\mb R)$ and $\|u_0-\phi_\omega\|_{H^1}<\delta$,
then the solution $u(t)$ of \eqref{gDNLS} with $u(0)=u_0$ exists for all $t\ge0$, 
and $u(t)\in U_\ve(\phi_\omega)$ for all $t\ge0$,
where
\[
U_\ve(\phi_\omega)
=\{\,u\in H^1(\mb R)\mid
\inf_{(s_0,s_1)\in\mb R^2}\|u-e^{is_0}\phi_\omega(\cdot-s_1)\|_{H^1}
<\ve\,\}.
\]
Otherwise, $T(\omega t)\phi$ is said to be \textit{unstable}.
\end{definition} 
%%%%%%%%%%%%%%%
For the case $\sigma=1$, 
Guo and Wu \cite{GW95} showed that the solitary wave $e^{i\omega_0t}\phi_\omega(\cdot-\omega_1t)$ is stable for $\omega\in\Omega$ with $\omega_1<0$,
and Colin and Ohta \cite{CO06} proved that the solitary wave $e^{i\omega_0t}\phi_\omega(\cdot-\omega_1t)$ is stable for all $\omega\in\Omega$.
For general exponents $\sigma>1$, 
Liu, Simpson and Sulem \cite{LSS13} proved that  for all $\sigma\ge2$ and $\omega\in\Omega$, 
the solitary wave $e^{i\omega_0t}\phi_\omega(\cdot-\omega_1t)$ is unstable.
In \cite{LSS13},
they also proved for $1<\sigma<2$ the solitary wave $e^{i\omega_0t}\phi_\omega(\cdot-\omega_1t)$ is stable if $-2\sqrt{\omega_0}<\omega_1<2z_0\sqrt{\omega_0}$, 
and unstable if $2z_0\sqrt{\omega_0}<\omega_1<2\sqrt{\omega_0}$,
where the constant $z_0=z_0(\sigma)\in(-1,1)$ is the solution of
\begin{align*}
F_\sigma(z)
&:=(\sigma-1)^2\left[
\int_0^\infty(\cosh y-z)^{-1/\sigma}\,dy
\right]^2 \\
&\qquad-\left[
\int_0^\infty(\cosh y-z)^{-1/\sigma-1}(z\cosh y-1)\,dy
\right]^2
=0.
\end{align*}
We note that $\det[d''(\omega)]$ has the same sign as 
$F_\sigma(\omega_1/2\sqrt{\omega_0})$ (see \cite[Lemma 4.2]{LSS13}).
In \cite{LSS13},
it is showed by numerical computation
that for $1<\sigma<2$, 
the function $F_\sigma$ is monotonically increasing,
$F_\sigma(-1)<0$ and $\lim_{z\uparrow1}F_\sigma(z)=+\infty$,
and that $F_\sigma$ has exactly one root $z_0$ in the interval $(-1,1)$.
The proof in \cite{LSS13} is based on the spectral analysis of the linearized operator $S_\omega''(\phi_\omega)$
and the Hessian matrix $d''(\omega)$,
and the general theory of Grillakis, Shatah and Strauss \cite{GSS2}.
However, the stability problem in the case 
$\omega_1=2z_0\sqrt{\omega_0}$ is open
because the Hessian matrix $d''(\omega)$ is degenerate.
While there are several papers treating the stability and instability of a one-parameter family of solitary waves in degenerate cases
(see \cite{CP03, Ohta11, CO12, Maeda12, Yamazaki15}),
to the best of our knowledge,
there are none for a two-parameter family of solitary waves.

In this paper, we consider the borderline case 
$\omega_1=2z_0\sqrt{\omega_0}$ 
and prove the following.

\begin{theorem} \label{mt}
Let $3/2\le\sigma<2$ and $z_0=z_0(\sigma)\in(-1,1)$ satisfy $F_\sigma(z_0)=0$.
Then the solitary wave $e^{i\omega_0t}\phi_\omega(\cdot-\omega_1t)$ is unstable if
$\omega_1=2z_0\sqrt{\omega_0}$.
\end{theorem}

\begin{remark}
If $3/2\le\sigma<2$, then $E\in C^3(H^1(\mb R),\mb R)$,
but if $1<\sigma<3/2$, $E$ is not $C^3$.
\end{remark}

The proof of Theorem \ref{mt} is based on similar arguments of Ohta \cite{Ohta11, Ohta14} and Maeda \cite{Maeda12}.
In our case $\omega_1=2z_0\sqrt{\omega_0}$,
it can be proved that if $\xi\in\mb R^2$ is an eigenvector of $d''(\omega)$ corresponding to the zero eigenvalue,
then $\frac{d^3}{d\lambda^3}d(\omega+\lambda\xi)|_{\lambda=0}\ne0$.
Hence, we can prove instability by the Lyapunov functional methods like \cite{Ohta11, Maeda12, Ohta14}.

The rest of this paper is organized as follows.
In section \ref{sec:sc}, 
we give a sufficient condition for instability in degenerate cases
and show that this condition holds in our cases.
In section \ref{sec:prf},
we prove this condition implies instability of the solitary wave 
$e^{i\omega_0t}\phi_\omega(\cdot-\omega_1t)$.

\section{Sufficient condition for instability} \label{sec:sc}

For $s=(s_0,s_1)\in\mb R^2$ and $v\in H^1(\mb R)$, we define
\[
T(s)v=T_0(s_0)T_1(s_1)v
=e^{is_0}v(\cdot-s_1).
\]
Then, the generator $T_j'(0)$ of $\{T_j(s)\}_{s\in\mb R}$ is given by
\[
T_0'(0)v=iv,\quad
T_1'(0)v=-\pt_xv,\quad
v\in H^1(\mb R).
\]
We define the bounded linear operator $B_j$ from $H^1(\mb R)$ to $L^2(\mb R)$ by
\[
B_jv
=-iT_j'(0)v.
\]
Then we have $Q_j'(v)=B_jv$.
Note that $E$ and $Q_j$ are invariant under $T$, that is,
\[
E(T(s)u)=E(u),\quad
Q_j(T(s)u)=Q_j(u).
\]
These and $S_\omega'(\phi_\omega)=0$ imply that $S_\omega''(\phi_\omega)T_j'(0)\phi_\omega=0$ for $j=0$, $1$.

For $\xi\in\mb R^2$, let
\[
B_\xi v
=\sum_{j=0}^1\xi_jB_jv,\quad
v\in H^1(\mb R).
\]

In this section,
we prove the following.

\begin{proposition} \label{ud}
Let $3/2\le\sigma<2$ and 
$z_0=z_0(\sigma)\in(-1,1)$ satisfy $F_\sigma(z_0)=0$.
Let $\omega_1=2z_0\sqrt{\omega_0}$ 
and $\xi=(\xi_0,\xi_1)\in\mb R^2$ be an eigenvector of the Hessian matrix $d''(\omega)$ corresponding to the zero eigenvalue.
Then there exists $\psi\in H^1(\mb R)$ with the following properties.
\begin{enumerate}
\item
$(B_j\phi_\omega,\psi)_{L^2}
=(T_j'(0)\phi_\omega,\psi)_{L^2}
=0$ for $j=0$, $1$,
and 
\[
S_\omega''(\phi_\omega)\psi
=-B_\xi\phi_\omega,\quad
\langle S_\omega'''(\phi_\omega)(\psi,\psi),\psi\rangle
\ne-3(B_\xi\psi,\psi)_{L^2}.
\]
\item
There exists $k_0>0$ such that 
$\langle S_\omega''(\phi_\omega)w,w\rangle
\ge k_0\|w\|_{H^1}^2$ for all $w\in H^1(\mb R)$ satisfying
\[
(w,B_\xi\phi_\omega)_{L^2}
=(w,T_0'(0)\phi_\omega)_{L^2}
=(w,T_1'(0)\phi_\omega)_{L^2}
=(w,\psi)_{L^2}
=0.
\]
\end{enumerate}
\end{proposition}

To prove Proposition \ref{ud} (i), 
we establish the following.

\begin{lemma} \label{d3}
Let $1<\sigma<2$ and 
$z_0=z_0(\sigma)\in(-1,1)$ satisfy $F_\sigma(z_0)=0$.
Let $\omega_1=2z_0\sqrt{\omega_0}$ 
and $\xi=(\xi_0,\xi_1)\in\mb R^2$ be an eigenvector of the Hessian matrix $d''(\omega)$ corresponding to the zero eigenvalue.
Then 
$\frac{d^3}{d\lambda^3}d(\omega+\lambda\xi)|_{\lambda=0}\ne0$.
\end{lemma}

The proof of Lemma \ref{d3} is given in Appendix \ref{sec:cul}.
To prove Proposition \ref{ud} (ii), 
we use the spectral condition of the linearized operator 
$S_\omega''(\phi_\omega)$ given by
\begin{align} \notag
S_\omega''(\phi_\omega)v
&=(-\pt_x^2
-i\sigma|\phi_\omega|^{2\sigma-2}\ol{\phi_\omega}\pt_x\phi_\omega
-i|\phi_\omega|^{2\sigma}\pt_x
+\omega_0
+\omega_1i\pt_x)v \\ \label{lo}
&\qquad-i\sigma|\phi_\omega|^{2\sigma-2}\phi_\omega\pt_x\phi_\omega \ol{v}.
\end{align}
The following result is due to \cite{LSS13}.

\begin{lemma}[{\cite[Theorem 3.1]{LSS13}}] \label{sc}
For $\sigma\ge1$ and $\omega\in\Omega$, there exist
$\chi_\omega\in H^1(\mb R)\setminus\{0\}$,
$\lambda_\omega<0$ and $k_1>0$ such that 
$S_\omega''(\phi_\omega)\chi_\omega=\lambda\chi_\omega$ and
$\langle S_\omega''(\phi_\omega)p,p\rangle
\ge k_1\|p\|_{L^2}^2$ for all $p\in H^1(\mb R)$ satisfying
\[
(p,\chi_\omega)_{L^2}
=(p,T_0'(0)\phi_\omega)_{L^2}
=(p,T_1'(0)\phi_\omega)_{L^2}
=0.
\]
\end{lemma}

\begin{remark}
By Lemma \ref{sc}, 
it is impossible that
for $\sigma\ge1$ and $\omega\in\Omega$,  
the Hessian matrix $d''(\omega)$ has two nonnegative eigenvalues (see \cite[Section 3]{GSS2}).
\end{remark}

Now, we verify Proposition \ref{ud}.
We define
\[
Q_\xi(u)
=\frac12(B_\xi u,u)_{L^2},\quad
u\in H^1(\mb R).
\]

\begin{proof}[Proof of Proposition \ref{ud}]
(i)
By differentiating $S_{\omega+\lambda\xi}'(\phi_{\omega+\lambda\xi})=0$ with respect to $\lambda$,
we have
\begin{align*}
&S_{\omega+\lambda\xi}''(\phi_{\omega+\lambda\xi})\pt_\lambda\phi_{\omega+\lambda\xi}
=-B_\xi\phi_{\omega+\lambda\xi}, \\
&S_{\omega+\lambda\xi}'''(\phi_{\omega+\lambda\xi})(\pt_\lambda\phi_{\omega+\lambda\xi},\pt_\lambda\phi_{\omega+\lambda\xi})
+S_{\omega+\lambda\xi}''(\phi_{\omega+\lambda\xi})\pt_\lambda^2\phi_{\omega+\lambda\xi}
=-2B_\xi\pt_\lambda\phi_{\omega+\lambda\xi}.
\end{align*}
By differentiating 
$d(\omega+\lambda\xi)
=S_{\omega+\lambda\xi}(\phi_{\omega+\lambda\xi})$
with respect to $\lambda$, we obtain
\begin{align*}
\frac{d}{d\lambda}d(\omega+\lambda\xi)
&=Q_\xi(\phi_{\omega+\lambda\xi}), \\
\frac{d^2}{d\lambda^2}d(\omega+\lambda\xi)
&=\langle Q_\xi'(\phi_{\omega+\lambda\xi}),\pt_\lambda\phi_{\omega+\lambda\xi}\rangle
=(B_\xi\phi_{\omega+\lambda\xi},\pt_\lambda\phi_{\omega+\lambda\xi})_{L^2}, \\
\frac{d^3}{d\lambda^3}d(\omega+\lambda\xi)
&=(B_\xi\pt_\lambda\phi_{\omega+\lambda\xi},\pt_\lambda\phi_{\omega+\lambda\xi})_{L^2}
+(B_\xi\phi_{\omega+\lambda\xi},\pt_\lambda^2\phi_{\omega+\lambda\xi})_{L^2} \\
&=(B_\xi\pt_\lambda\phi_{\omega+\lambda\xi},\pt_\lambda\phi_{\omega+\lambda\xi})_{L^2}
-\langle S_{\omega+\lambda\xi}''(\phi_{\omega+\lambda\xi})\pt_\lambda\phi_{\omega+\lambda\xi},\pt_\lambda^2\phi_{\omega+\lambda\xi}\rangle \\
&=(B_\xi\pt_\lambda\phi_{\omega+\lambda\xi},\pt_\lambda\phi_{\omega+\lambda\xi})_{L^2}
-\langle S_{\omega+\lambda\xi}''(\phi_{\omega+\lambda\xi})\pt_\lambda^2\phi_{\omega+\lambda\xi},\pt_\lambda\phi_{\omega+\lambda\xi}\rangle \\
&=\langle S_{\omega+\lambda\xi}'''(\phi_{\omega+\lambda\xi})(\pt_\lambda\phi_{\omega+\lambda\xi},\pt_\lambda\phi_{\omega+\lambda\xi}),\pt_\lambda\phi_{\omega+\lambda\xi}\rangle \\
&\quad+3(B_\xi\pt_\lambda\phi_{\omega+\lambda\xi},\pt_\lambda\phi_{\omega+\lambda\xi})_{L^2}.
\end{align*}
We take 
\[
\hat\psi
=\pt_\lambda\phi_{\omega+\lambda\xi}|_{\lambda=0}
=\xi_0\pt_{\omega_0}\phi_\omega+\xi_1\pt_{\omega_1}\phi_\omega.
\]
Then, since $\xi$ is a zero eigenvector of $d''(\omega)$,
we deduce
\[
0=d''(\omega)\xi
=\begin{bmatrix}
\langle Q_0'(\phi_\omega),\xi_0\pt_{\omega_0}\phi_\omega+\xi_1\pt_{\omega_1}\phi_\omega\rangle \\
\langle Q_1'(\phi_\omega),\xi_0\pt_{\omega_0}\phi_\omega+\xi_1\pt_{\omega_1}\phi_\omega\rangle
\end{bmatrix}
=\begin{bmatrix}
(B_0\phi_\omega,\hat\psi)_{L^2} \\
(B_1\phi_\omega,\hat\psi)_{L^2}
\end{bmatrix}.
\]
Moreover, it follows from
$\frac{d^3}{d\lambda^3}d(\omega+\lambda\xi)|_{\lambda=0}\ne0$
that
\[
0\ne\frac{d^3}{d\lambda^3}d(\omega+\lambda\xi)|_{\lambda=0}
=\langle S_\omega'''(\phi_\omega)(\hat\psi,\hat\psi),\hat\psi\rangle
+3(B_\xi\hat\psi,\hat\psi)_{L^2}.
\]
Let
\[
\psi
=\hat\psi+\sum_{j=0}^1\mu_jT_j'(0)\phi_\omega,
\]
where $(\mu_0,\mu_1)\in\mb{R}^2$ is taken so that
\[
(T_0'(0)\phi_\omega,\psi)_{L^2}
=(T_1'(0)\phi_\omega,\psi)_{L^2}=0.
\]
Then we see that
\begin{align*}
&(B_j\phi_\omega,\psi)_{L^2}
=(B_j\phi_\omega,\hat\psi)_{L^2}
=0,\quad j=0,1, \\
&S_\omega''(\phi_\omega)\psi
=S_\omega''(\phi_\omega)\hat\psi
=-B_\xi\phi_\omega, \\
&\langle S_\omega'''(\phi_\omega)(\psi,\psi),\psi\rangle
+3(B_\xi\psi,\psi)_{L^2}
=\langle S_\omega'''(\phi_\omega)(\hat\psi,\hat\psi),\hat\psi\rangle
+3(B_\xi\hat\psi,\hat\psi)_{L^2}
\ne0.
\end{align*}
Thus, we have the conclusion.

(ii)
First, we show that
$\langle S_\omega''(\phi_\omega)w,w\rangle\ge k_1\|w\|_{L^2}^2$,
where $k_1$ is the positive constant given in Lemma \ref{sc}.
Since $\psi\ne0$ and $(\psi,T_j'(0)\phi_\omega)_{L^2}
=\langle S_\omega''(\phi_\omega)\psi,\psi\rangle=0$,
it follows from Lemma \ref{sc} that
$(\psi,\chi_\omega)_{L^2}\ne0$.
Let 
\[
\alpha=-\frac{(w,\chi_\omega)_{L^2}}{(\psi,\chi_\omega)_{L^2}},\quad
p=w+\alpha\psi.
\]
Then we have 
$(p,\chi_\omega)_{L^2}
=(p,T_0'(0)\phi_\omega)_{L^2}
=(p,T_1'(0)\phi_\omega)_{L^2}
=0$.
By Lemma \ref{sc} and $(w,\psi)_{L^2}=0$, we obtain
\[
\langle S_\omega''(\phi_\omega)p,p\rangle
\ge k_1\|w+\alpha\psi\|_{L^2}^2
\ge k_1\|w\|_{L^2}^2.
\]
On the other hand, by
$S_\omega''(\phi_\omega)\psi
=-B_\xi\phi_\omega$ and
$\langle S_\omega''(\phi_\omega)\psi,\psi\rangle
=(w,B_\xi\phi_\omega)_{L^2}=0$, we have
$\langle S_\omega''(\phi_\omega)p,p\rangle
=\langle S_\omega''(\phi_\omega)w,w\rangle$.

Next, we prove the conclusion.
Since $\phi_\omega$, $\pt_x\phi_\omega\in L^\infty(\mb R)$,
by \eqref{lo}, we see that there exist positive constants $c$ and $C$ such that
\begin{align*}
c\|v\|_{H^1}^2
\le\langle S_\omega''(\phi_\omega)v,v\rangle
+C\|v\|_{L^2}^2
\end{align*}
for all $v\in H^1(\mb R)$. 
This and first claim imply the conclusion.
\end{proof}

\section{Proof of Theorem \ref{mt}} \label{sec:prf}

In this section, we prove Theorem \ref{mt} by using Proposition \ref{ud}.
Throughout this section, 
let $3/2\le\sigma<2$,
$z_0=z_0(\sigma)\in(-1,1)$ satisfy $F_\sigma(z_0)=0$, 
$\omega_1=2z_0\sqrt{\omega_0}$ and
$\xi=(\xi_0,\xi_1)\in\mb R^2$ be an eigenvector of the Hessian matrix $d''(\omega)$ corresponding to the zero eigenvalue.

\begin{lemma} \label{s1}
There exist $\lambda_0>0$ and a $C^\infty$-mapping 
$\rho\colon(-\lambda_0,\lambda_0)\to\mb{R}$ such that
\begin{equation} \label{Qxi}
Q_\xi(\phi_\omega+\lambda\psi+\rho(\lambda)B_\xi\phi_\omega)
=Q_\xi(\phi_\omega)
\end{equation}
for all $\lambda\in(-\lambda_0,\lambda_0)$, and
\begin{equation} \label{r1}
\rho(\lambda)
=-\frac{(B_\xi\psi,\psi)_{L^2}}{2\|B_\xi\phi_\omega\|_{L^2}^2}\lambda^2
+o(\lambda^2)
\end{equation}
as $\lambda\to 0$.
\end{lemma}

\begin{proof}
We define
\[
F(\lambda,\rho)
=Q_\xi(\phi_\omega+\lambda\psi+\rho B_\xi\phi_\omega)
-Q_\xi(\phi_\omega),\quad
(\lambda,\rho)\in\mb{R}^2.
\]
Then we have
$F(0,0)=0$ and
\[\pt_\rho F(0,0)
=\langle Q_\xi'(\phi_\omega),B_\xi\phi_\omega\rangle
=\|B_\xi\phi_\omega\|_H^2
\ne0.
\]
By the implicit function theorem, 
there exist $\lambda_0>0$ and a $C^\infty$-mapping
$\rho\colon(-\lambda_0,\lambda_0)\to\mb{R}$ such that
$F(\lambda,\rho(\lambda))=0$ for all $\lambda\in(-\lambda_0,\lambda_0)$.

Moreover, by differentiating $F(\lambda,\rho(\lambda))=0$ with respect to $\lambda$, we obtain
\[
\rho'(0)
=0,\quad
\rho''(0)
=-\frac{(B_\xi\psi,\psi)_{L^2}}{\|B_\xi\phi_\omega\|_{L^2}^2}.
\]
This completes the proof.
\end{proof}

We define
\[
\Psi(\lambda)
=\phi_\omega
+\lambda\psi
+\rho(\lambda)B_\xi\phi_\omega,\quad
\lambda\in(-\lambda_0,\lambda_0).
\]

\begin{lemma} \label{s2}
There exist $\ve_0>0$ and
$C^3$-mappings 
$\alpha=(\alpha_0,\alpha_1)\colon U_{\ve_0}(\phi_\omega)\to\mb{R}^2$, 
$\Lambda\colon U_{\ve_0}(\phi_\omega)\to(-\lambda_0,\lambda_0)$,
$\beta\colon U_{\ve_0}(\phi_\omega)\to\mb{R}$,
$w\colon U_{\ve_0}(\phi_\omega)\to H^1(\mb R)$ such that
\begin{gather} \label{Tau}
T(\alpha(u))u
=\Psi(\Lambda(u))
+\beta(u)B_\xi\phi_\omega
+w(u), \\ \notag
(w(u),B_\xi\phi_\omega)_{L^2}
=(w(u),T_0'(0)\phi_\omega)_{L^2}
=(w(u),T_1'(0)\phi_\omega)_{L^2}
=(w(u),\psi)_{L^2}
=0, \\ \notag
\alpha(T(s)u)
=\alpha(u)-s,\quad 
\Lambda(T(s)u)
=\Lambda(u), \\ \notag
\beta(T(s)u)
=\beta(u),\quad
w(T(s)u)
=w(u)
\end{gather}
for all $u\in U_{\ve_0}(\phi_\omega)$ and $s\in\mb{R}^2$.
\end{lemma}

\begin{proof}
We define 
\[
G(u,\alpha,\Lambda,\beta)
=\begin{bmatrix}
(T(\alpha)u-\Psi(\Lambda)-\beta B_\xi\phi_\omega,T_0'(0)\phi_\omega)_{L^2} \\
(T(\alpha)u-\Psi(\Lambda)-\beta B_\xi\phi_\omega,T_1'(0)\phi_\omega)_{L^2} \\
(T(\alpha)u-\Psi(\Lambda)-\beta B_\xi\phi_\omega,\psi)_{L^2} \\
(T(\alpha)u-\Psi(\Lambda)-\beta B_\xi\phi_\omega,B_\xi\phi_\omega)_{L^2}
\end{bmatrix},
\]
for 
$(u,\alpha,\Lambda,\beta)\in H^1(\mb R)\times\mb{R}^2\times\mb{R}\times\mb{R}$.
Then, we have
$G(\phi_\omega,0,0,0)=0$ and 
\begin{align*}
&\frac{\pt G}{\pt(\alpha,\Lambda,\beta)}(\phi_\omega,0,0,0) \\
&=\begin{bmatrix}
\|T_0'(0)\phi_\omega\|_{L^2}^2 & (T_1'(0)\phi_\omega,T_0'(0)\phi_\omega)_{L^2} & 0 &0 \\
(T_0'(0)\phi_\omega,T_1'(0)\phi_\omega)_{L^2} & \|T_1'(0)\phi_\omega\|_{L^2}^2 &0 & 0 \\
0 & 0 & -\|\psi\|_{L^2}^2 & 0 \\
0 & 0 & 0 & -\|B_\xi\phi_\omega\|_{L^2}^2
\end{bmatrix}.
\end{align*}
Since $T_0'(0)\phi_\omega$, $T_1'(0)\phi_\omega$ are linearly independent, we see that 
$\frac{\pt G}{\pt(\alpha,\Lambda,\beta)}(\phi_\omega,0,0,0)$
is invertible.
Thus by the implicit function theorem, there exist $\ve_0>0$, 
$\alpha=(\alpha_0,\alpha_1)\colon U_{\ve_0}(\phi_\omega)\to\mb{R}^2$,
$\Lambda\colon U_{\ve_0}(\phi_\omega)\to(-\lambda_0,\lambda_0)$ and
$\beta\colon U_{\ve_0}(\phi_\omega)\to\mb{R}$ such that 
$G(u,\alpha(u),\Lambda(u),\beta(u))=0$ for all $u\in U_{\ve_0}(\phi_\omega)$.
Finally, we define
\[
w(u)
=T(\alpha(u))u
-\Psi(\Lambda(u))
-\beta(u)B_\xi\phi_\omega,\quad
u\in U_{\ve_0}(\phi_\omega).
\]
Then, we have the conclusion. 
\end{proof}

\begin{remark}
By the uniqueness of the solution of $G=0$, we have
\[
\alpha(\Psi(\lambda))=0,\
\Lambda(\Psi(\lambda))=\lambda,\
\beta(\Psi(\lambda))=0,\
w(\Psi(\lambda))=0,\quad
\lambda\in(-\lambda_0,\lambda_0).
\]
\end{remark}

\begin{lemma} \label{s3}
$\alpha_j'(u)$, $\Lambda'(u)$, $\alpha_j''(u)v\in H^1(\mb R)$ for all $u\in U_{\ve_0}(\phi_\omega)$ and $v\in H^1(\mb R)$. 
Moreover
\begin{equation} \label{doa}
\begin{aligned}
\alpha_0'(\phi_\omega)
&=\frac{-\|T_1'(0)\phi_\omega\|_{L^2}^2T_0'(0)\phi_\omega+(T_0'(0)\phi_\omega,T_1'(0)\phi_\omega)_{L^2}T_1'(0)\phi_\omega}
{\|T_0'(0)\phi_\omega\|_{L^2}^2\|T_1'(0)\phi_\omega\|_{L^2}^2-(T_0'(0)\phi_\omega,T_1'(0)\phi_\omega)_{L^2}^2}, \\
\alpha_1'(\phi_\omega)
&=\frac{(T_0'(0)\phi_\omega,T_1'(0)\phi_\omega)_{L^2}T_0'(0)\phi_\omega-\|T_0'(0)\phi_\omega\|_{L^2}^2T_1'(0)\phi_\omega}
{\|T_0'(0)\phi_\omega\|_{L^2}^2\|T_1'(0)\phi_\omega\|_{L^2}^2-(T_0'(0)\phi_\omega,T_1'(0)\phi_\omega)_{L^2}^2}.
\end{aligned}
\end{equation}
\end{lemma}

\begin{proof}
By differentiating $G(u,\alpha(u),\Lambda(u),\beta(u))=0$ with respect to $u$, 
we have
\begin{align} \label{doa2}
\begin{bmatrix}
\alpha_0'(u) \\ \alpha_1'(u) \\ \Lambda'(u) \\ \beta'(u)
\end{bmatrix}
&=-\left[
\frac{\pt G}{\pt(\alpha,\Lambda,\mu)}(u,\alpha(u),\Lambda(u),\beta(u))
\right]^{-1}
\begin{bmatrix}
T(-\alpha(u))T_0'(0)\phi_\omega \\ 
T(-\alpha(u))T_1'(0)\phi_\omega \\ 
T(-\alpha(u))\psi \\ 
T(-\alpha(u))B_\xi\phi_\omega
\end{bmatrix} \\ \notag
&\in H^1(\mb R)^4.
\end{align}
Similarly, we also see that $\alpha_j''(u)v\in H^1(\mb R)$.
Moreover, by substituting $\phi_\omega$ for $u$ in \eqref{doa2},
we obtain \eqref{doa}.
\end{proof}

\begin{lemma} \label{s4}
For $u\in U_{\ve_0}(\phi_\omega)$ satisfying $Q_\xi(u)=Q_\xi(\phi_\omega)$,
\[
\beta(u)
=O(|\Lambda(u)|\|w(u)\|_{H^1}+\|w(u)\|_{H^1}^2)
\]
as $\inf_{s\in\mb{R}^2}\|u-T(s)\phi_\omega\|_{H^1}\to0$.
\end{lemma}

\begin{proof}
For $u\in U_{\ve_0}(\phi_\omega)$ satisfying $Q_\xi(u)=Q_\xi(\phi_\omega)$, 
by \eqref{Tau}, \eqref{Qxi} and $(B_\xi\phi_\omega,w(u))_{L^2}=0$,
we have
\begin{align*}
0&=Q_\xi(u)-Q_\xi(\phi_\omega)
=Q_\xi(T(\alpha(u))u)-Q_\xi(\phi_\omega) \\
&=Q_\xi(\Psi(\Lambda(u))+\beta(u)B_\xi\phi_\omega+w(u))-Q_\xi(\phi_\omega) \\
&=Q_\xi(\Psi(\Lambda(u)))-Q_\xi(\phi_\omega)
+\beta(u)^2Q_\xi(B_\xi\phi_\omega)
+Q_\xi(w(u)) \\
&\quad+\beta(u)(B_\xi\Psi(\Lambda(u)),B_\xi\phi_\omega)_{L^2}
+\beta(u)(B_\xi^2\phi_\omega,w(u))_{L^2}
+(B_\xi\Psi(\Lambda(u)),w(u))_{L^2} \\
&=\beta(u)[\|B_\xi\phi_\omega\|_{L^2}^2+o(1)]
+Q_\xi(w(u))
+O(|\Lambda(u)|\|w(u)\|_{L^2}).
\end{align*}
This finishes the proof.
\end{proof}

For $u\in U_{\ve_0}(\phi_\omega)$, we define
\begin{align*}
M(u)
=T(\alpha(u))u,\quad
A(u)
=-(M(u),i\psi)_{L^2}.
\end{align*}
Then 
\begin{align*}
A'(u)
&=-\sum_{j=0}^1(T_j'(0)M(u),i\psi)_{L^2}\alpha_j'(u)
-iT(-\alpha(u))\psi, \\
A''(u)v
&=-\sum_{j=0}^1(T_j'(0)M(u),i\psi)_{L^2}\alpha_j''(u)v \\
&\quad-\sum_{j=0}^1\langle\alpha_j'(u),v\rangle
\sum_{k=0}^1(T_k'(0)T_j'(0)M(u),i\psi)_{L^2}\alpha_k'(u) \\
&\quad-\sum_{j=0}^1\langle\alpha_j'(u),v\rangle T(-\alpha(u))B_j\psi
-\sum_{j=0}^1(T_j'(0)T(\alpha(u))v,i\psi)_{L^2}\alpha_j'(u).
\end{align*}
By Lemma \ref{s3}, 
we see that $A'(u)$, $A''(u)v\in H^1(\mb R)$ 
for all $u\in U_{\ve_0}(\phi_\omega)$ and $v\in H^1(\mb R)$.
Moreover, by \eqref{doa}, we deduce
\begin{align} \label{A1}
iA'(\phi_\omega)
&=\psi, \\ \label{A2}
iA''(\phi_\omega)\psi
&=-\sum_{j=0}^1(B_j\psi,\psi)_{L^2}i\alpha_j'(\phi_\omega).
\end{align}
Since $M$ and $A$ are invariant under $T$,
it follows that
\begin{equation} \label{PQ}
0
=\frac{d}{ds}A(T_j(s)u)|_{s=0}
=\langle A'(u),T_j'(0)u\rangle
=-\langle Q'_j(u),iA'(u)\rangle.
\end{equation}
For $u\in U_{\ve_0}(\phi_\omega)$, we define
\[
P(u)
=\langle E'(u),iA'(u)\rangle.
\]
Then by \eqref{PQ}, we have
\[
P(u)
=\langle S_\omega'(u),iA'(u)\rangle.
\]
By $S_\omega'(\phi_\omega)=0$, \eqref{A1} and
$S_\omega''(\phi_\omega)\psi
=-B_\xi\phi_\omega$,
we obtain
\begin{equation} \label{P1}
P'(\phi_\omega)
=-B_\xi\phi_\omega.
\end{equation}
Moreover, by \eqref{A2} and \eqref{doa}, we deduce
\begin{align} \notag
\langle P''(\phi_\omega)\psi,\psi\rangle
&=\langle S_\omega'''(\phi_\omega)(\psi,\psi),\psi\rangle
+2\langle S_\omega''(\phi_\omega)\psi,iA''(\phi_\omega)\psi\rangle \\ \label{P2}
&=\langle S_\omega'''(\phi_\omega)(\psi,\psi),\psi\rangle
+2(B_\xi\psi,\psi)_{L^2}.
\end{align}
We note that $P$ is invariant under $T$.

\begin{lemma} \label{AP}
Let $I$ be an interval of $\mb{R}$.
Let $u\in C(I,H^1(\mb R))\cap C^1(I,H^{-1}(\mb R))$ be a solution of \eqref{gDNLS},
and assume that $u(t)\in U_{\ve_0}(\phi_\omega)$ for all $t\in I$.
Then,
\[
\frac{d}{dt}A(u(t))
=P(u(t))
\]
for all $t\in I$.
\end{lemma}

\begin{proof}
By Lemma 4.6 of \cite{GSS1}, 
we see that $t\mapsto A(u(t))$ is $C^1$ on $I$, and
\begin{align*}
\frac{d}{dt}A(u(t))
=\left\langle\pt_tu(t),A'(u(t))\right\rangle
=\langle E'(u(t)),iA'(u(t))\rangle
=P(u(t))
\end{align*}
for all $t\in I$.
This completes the proof.
\end{proof}

Let
\[
\nu
=\langle S_\omega'''(\phi_\omega)(\psi,\psi),\psi\rangle
+3(B_\xi\psi,\psi)_{L^2}.
\]
Then $\nu\ne0$ by Proposition \ref{ud} (i).

\begin{lemma}
For $\lambda\in(-\lambda_0,\lambda_0)$,
\begin{align} \label{Sl}
&S_\omega(\Psi(\lambda))
-S_\omega(\phi_\omega)
=\frac{\lambda^3}{6}\nu+o(\lambda^3), \\ \label{Pl}
&P(\Psi(\lambda))
=\frac{\lambda^2}2\nu+o(\lambda^2)
\end{align}
as $\lambda\to0$.
\end{lemma}

\begin{proof}
Since
$S_\omega'(\phi_\omega)=0$,
$S_\omega''(\phi_\omega)\psi=-B_\xi\phi_\omega$ and
$(\psi,B_\xi\phi_\omega)_{L^2}=0$,
by Taylor's expansion and \eqref{r1},
we have
\begin{align*}
&S_\omega(\Psi(\lambda))
-S_\omega(\phi_\omega)
=S_\omega(\phi_\omega+\lambda\psi+\rho(\lambda)B_\xi\phi_\omega)
-S_\omega(\phi_\omega) \\
&=\langle 
S_\omega'(\phi_\omega),
\lambda\psi
+\rho(\lambda)B_\xi\phi_\omega
\rangle
+\frac{\lambda^2}{2}\langle S_\omega''(\phi_\omega)\psi,\psi\rangle
+\lambda\rho(\lambda)\langle S_\omega''(\phi_\omega)\psi,B_\xi\phi_\omega\rangle \\
&\quad+\frac{\lambda^3}6\langle S_\omega'''(\phi_\omega)(\psi,\psi),\psi\rangle
+o(\lambda^3) \\
&=\frac{\lambda^3}6[
\langle S_\omega'''(\phi_\omega)(\psi,\psi),\psi\rangle
+3(B_\xi\psi,\psi)_{L^2}
]
+o(\lambda^3).
\end{align*}
On the other hand,
by \eqref{r1}, \eqref{P1} and \eqref{P2}, we have
\begin{align*}
&P(\Psi(\lambda))
=P(\phi_\omega+\lambda\psi+\rho(\lambda)B_\xi\phi_\omega) \\
&=P(\phi_\omega)
+\langle P'(\phi_\omega),\lambda\psi+\rho(\lambda)B_\xi\phi_\omega\rangle
+\frac{\lambda^2}2\langle P''(\phi_\omega)\psi,\psi\rangle
+o(\lambda^2) \\
&=\frac{\lambda^2}2[\langle S_\omega'''(\phi_\omega)(\psi,\psi),\psi\rangle
+3(B_\xi\psi,\psi)_{L^2}]
+o(\lambda^2).
\end{align*}
This completes the proof.
\end{proof}

\begin{lemma} \label{SPu}
For $u\in U_{\ve_0}(\phi_\omega)$ satisfying 
$Q_\xi(u)=Q_\xi(\phi_\omega)$,
\begin{align}
&S_\omega(u)
-S_\omega(\phi_\omega)
=\frac{\Lambda(u)^3}6\nu
+\frac12\langle S_\omega''(\phi_\omega)w(u),w(u)\rangle
+o(|\Lambda(u)|^3+\|w(u)\|_{H^1}^2), \\
&\Lambda(u)P(u)
=\frac{\Lambda(u)^3}{2}\nu
+o(|\Lambda(u)|^3+\|w(u)\|_{H^1}^2)
\end{align}
as $\inf_{s\in\mb{R}^2}\|u-T(s)\phi_\omega\|_{H^1}\to0$.
\end{lemma}

\begin{proof}
By Lemmas \ref{s2}, \ref{s4} and \eqref{Sl}, 
we have
\begin{align*}
&S_\omega(u)
-S_\omega(\phi_\omega) \\
&=S_\omega(M(u))
-S_\omega(\phi_\omega)
=S_\omega(\Psi(\Lambda(u))+\beta(u)B_\xi\phi_\omega+w(u))
-S_\omega(\phi_\omega) \\
&=S_\omega(\Psi(\Lambda(u)))
-S_\omega(\phi_\omega)
+\langle S_\omega'(\Psi(\Lambda(u))),\beta(u)B_\xi\phi_\omega+w(u)\rangle \\
&\quad+\frac12\langle S_\omega''(\Psi(\Lambda(u)))(\beta(u)B_\xi\phi_\omega+w(u)),\beta(u)B_\xi\phi_\omega+w(u)\rangle \\
&\quad+o(\|\beta(u)B_\xi\phi_\omega+w(u)\|_{H^1}^2) \\
&=\frac{\Lambda(u)^3}6\nu
+\Lambda(u)\langle S_\omega''(\phi_\omega)\psi,\beta(u)B_\xi\phi_\omega+w(u)\rangle
+\frac12\langle S_\omega''(\phi_\omega)w(u),w(u)\rangle \\
&\quad+o(|\Lambda(u)|^3+\|w(u)\|_{H^1}^2) \\
&=\frac{\Lambda(u)^3}6\nu
+\frac12\langle S_\omega''(\phi_\omega)w(u),w(u)\rangle
+o(|\Lambda(u)|^3+\|w(u)\|_{H^1}^2).
\end{align*}
On the other hand,
by Lemmas \ref{s2}, \ref{s4} and \eqref{Pl},
we deduce
\begin{align*}
&P(u) \\
&=P(M(u))
=P(\Psi(\Lambda(u))+\beta(u)B_\xi\phi_\omega+w(u)) \\
&=P(\Psi(\Lambda(u)))
+\langle P'(\Psi(\Lambda(u))),\beta(u)B_\xi\phi_\omega+w(u)\rangle \\
&\quad+O(\|\beta(u)B_\xi\phi_\omega+w(u)\|_{H^1}^2) \\
&=\frac{\Lambda(u)^2}{2}\nu
+\langle P'(\phi_\omega),w(u)\rangle
+O(|\Lambda(u)|^3+\|w(u)\|_{H^1}^2+|\Lambda(u)|\|w(u)\|_{H^1}) \\
&=\frac{\Lambda(u)^2}{2}\nu
+\langle S_\omega''(\phi_\omega)\psi,w(u)\rangle \\
&\quad+O(|\Lambda(u)|^3+\|w(u)\|_{H^1}^2+|\Lambda(u)|\|w(u)\|_{H^1}) \\
&=\frac{\Lambda(u)^2}{2}\nu
+O(|\Lambda(u)|^3+\|w(u)\|_{H^1}^2+|\Lambda(u)|\|w(u)\|_{H^1}).
\end{align*}
This completes the proof.
\end{proof}

\begin{proof}[Proof of Theorem \ref{mt}]
Without loss of generality, we may assume that $\nu>0$.
Note that by \eqref{Sl}, 
there exists $\lambda_1\in(0,\lambda_0)$ such that 
$S_\omega(\phi_\omega)-S_\omega(\Psi(\lambda))>0$ for all 
$\lambda\in(-\lambda_1,0)$.
By Lemma \ref{SPu} and Proposition \ref{ud} (ii), 
we see that there exist $\ve_1\in(0,\ve_0)$ and $c>0$ such that
\begin{equation} \label{f2}
S_\omega(u)
-S_\omega(\phi_\omega)
-\Lambda(u)P(u)
\ge -c\Lambda(u)^3
\end{equation}
for all $u\in U_{\ve_1}(\phi_\omega)$.

Suppose that $T(\omega t)\phi_\omega$ is stable.
Let $u_\lambda(t)$ be the solution of \eqref{gDNLS} with 
$u_\lambda(0)=\Psi(\lambda)$.
Since $T(\omega t)\phi_\omega$ is stable, 
there exists $\lambda_2\in(0,\lambda_1)$ such that
$u_\lambda(t)\in U_{\ve_1}(\phi_\omega)$
for all $\lambda\in(-\lambda_2,\lambda_2)$ and $t\ge0$.
Let $\lambda\in(-\lambda_2,0)$.
Then by the conservation of $S_\omega$ and \eqref{f2}, we have
\begin{align*}
0&<\delta_\lambda
:=S_\omega(\phi_\omega)
-S_\omega(u_\lambda(0)) \\
&=S_\omega(\phi_\omega)
-S_\omega(u_\lambda(t))
\le c\Lambda(u_\lambda(t))^3
-\Lambda(u_\lambda(t))P(u_\lambda(t))
\end{align*}
for all $t\ge0$.
By this inequality, $\Lambda(u_\lambda(0))=\lambda<0$ and continuity of $t\mapsto\Lambda(u_\lambda(t))$,
we see that $\Lambda(u_\lambda(t))<0$ for all $t\ge0$.
Thus, we have $\delta_\lambda<\lambda_0P(u_\lambda(t))$ for all $t\ge0$.
Moreover, by Lemma \ref{AP}, we have
\[
\frac{d}{dt}A(u_\lambda(t))
=P(u_\lambda(t)),\quad
t\ge0.
\]
Therefore, we see that 
$A(u_\lambda(t))\to\infty$ as $t\to+\infty$.
This contradicts the fact that there exists $C>0$ such that
$|A(u)|\le C$
for all $u\in U_{\ve_0}(\phi_\omega)$.

Hence, $T(\omega t)\phi_\omega$ is unstable.
\end{proof}

\appendix
\section{Proof of Lemma \ref{d3}} \label{sec:cul}

In this section, 
we prove Lemma \ref{d3}.
Throughout this section, 
let $1<\sigma<2$ and 
$z_0=z_0(\sigma)\in(-1,1)$ satisfy $F_\sigma(z_0)=0$.
For $\omega\in\Omega$,
we define
\begin{align*}
\kappa_\omega
&=\sqrt{4\omega_0-\omega_1^2}, \\
\tilde\kappa_\omega
&=2^{1/\sigma-2}\sigma^{-1}(1+\sigma)^{1/\sigma}\kappa_\omega^{2/\sigma-2}\omega_0^{-1/2\sigma-1/2}.
\end{align*}
Then we have
\begin{align} \label{ptk}
&\pt_{\omega_0}\kappa_\omega
=\frac{2}{\kappa_{\omega}},\quad
%%%%%%%%%%%%%%%%%%%%
\pt_{\omega_1}\kappa_\omega
=-\frac{\omega_1}{\kappa_\omega}, \\
%%%%%%%%%%%%%%%%%%%%
&\pt_{\omega_0}\tilde\kappa_\omega
=\tilde\kappa_\omega\left[
\frac{4(1-\sigma)}{\sigma\kappa_\omega^2}
-\frac{1+\sigma}{2\sigma\omega_0}
\right],\quad
\pt_{\omega_1}\tilde\kappa_\omega
=-\tilde\kappa_\omega\frac{2(1-\sigma)\omega_1}{\sigma\kappa_\omega^2}.
\end{align}
For $\omega\in\Omega$ and $n\in\mb Z_+$, we define
\[
\alpha_{n,\omega}
=\int_0^\infty\left(
\cosh(\sigma\kappa_\omega x)
-\frac{\omega_1}{2\sqrt{\omega_0}}
\right)^{-1/\sigma-n}\,dx.
\]
Then it follows from \cite[Lemmas A.1 and A.2]{LSS13} that
\begin{align}
\pt_{\omega_0}\alpha_{0,\omega}
&=-\frac{2}{\kappa_\omega^2}\alpha_{0,\omega}
-\frac{\omega_1}{4\sigma\omega_0^{3/2}}\alpha_{1,\omega}, \\
%%%%%%%%%%%%%%%%%%%%
\pt_{\omega_1}\alpha_{0,\omega}
&=\frac{\omega_1}{\kappa_\omega^2}\alpha_{0,\omega}
+\frac{1}{2\sigma\sqrt{\omega_0}}\alpha_{1,\omega},\\
%%%%%%%%%%%%%%%%%%%%
\pt_{\omega_0}\alpha_{1,\omega}
&=-\frac{\omega_1}{\sigma\sqrt{\omega_0}\kappa_\omega^2}\alpha_{0,\omega}
-\frac{\omega_1^2(2+\sigma)+4\sigma\omega_0}{2\sigma\omega_0\kappa_\omega^2}\alpha_{1,\omega},\\ \label{pta}
%%%%%%%%%%%%%%%%%%%%
\pt_{\omega_1}\alpha_{1,\omega}
&=\frac{2\sqrt{\omega_0}}{\sigma\kappa_\omega^2}\alpha_{0,\omega}
+\frac{2\omega_1(\sigma+1)}{\sigma\kappa_\omega^2}\alpha_{1,\omega}.
\end{align}

By \cite[Lemma A.3]{LSS13} and \eqref{Hd},
we obtain
\begin{align} \label{d21}
&\pt_{\omega_0}^2d(\omega)
=\frac{\tilde\kappa_\omega}{\omega_0^{1/2}}
(2\omega_1^2-8(\sigma-1)\omega_0)\alpha_{0,\omega}
-\frac{\tilde\kappa_\omega}{\omega_0}\kappa_\omega^2\omega_1\alpha_{1,\omega}
=\frac{\pt_{\omega_1}^2d(\omega)}{\omega_0}, \\
%%%%%%%%%%%%%%%%%%%%
&\pt_{\omega_1}\pt_{\omega_0}d(\omega) \label{d22}
=4\tilde\kappa_\omega
\omega_1(\sigma-2)\omega_0^{1/2}\alpha_{0,\omega}
+2\tilde\kappa\kappa_\omega^2\alpha_{1,\omega}
=\pt_{\omega_0}\pt_{\omega_1}d(\omega).
\end{align}
Eq.\ \eqref{d21} implies that
\begin{align} \label{d31}
\omega_0\pt_{\omega_0}^3d(\omega)
&=\pt_{\omega_0}\pt_{\omega_1}^2d(\omega)
-\pt_{\omega_0}^2d(\omega), \\ \label{d32}
%%%%%%%%%%%%%%%%%%%%
\pt_{\omega_1}^3d(\omega)
&=\omega_0\pt_{\omega_0}^2\pt_{\omega_1}d(\omega).
\end{align}
On the other hand, by differentiating \eqref{d22}, it follows from \eqref{ptk}--\eqref{pta} that
\begin{align} \notag
\pt_{\omega_0}^2\pt_{\omega_1}d(\omega)
&=\frac{2\omega_1\tilde\kappa_\omega\alpha_{0,\omega}}{\sigma\kappa_\omega^2\omega_0^{1/2}}[
4(2-3\sigma)(\sigma-2)\omega_0
-(\sigma-1)\kappa_\omega^2
] \\ \label{d33}
&\quad+\frac{\tilde\kappa_\omega\alpha_{1,\omega}}{\sigma\omega_0}[
4(2-\sigma)\omega_0
-2\sigma\omega_1^2
-(1+\sigma)\kappa_\omega^2
], \\ \notag
%%%%%%%%%%
\pt_{\omega_0}\pt_{\omega_1}^2d(\omega)
&=\frac{4\omega_0^{1/2}\tilde\kappa_\omega\alpha_{0,\omega}}{\sigma\kappa_\omega^2}[
(3\sigma-2)(\sigma-2)\omega_1^2
+(\sigma-1)^2\kappa_\omega^2
] \\ \label{d34}
&\quad+\frac{2(3\sigma-2)\omega_1\tilde\kappa_\omega\alpha_{1,\omega}}{\sigma}.
\end{align}

Let $\omega_1=2z_0\sqrt{\omega_0}$. 
Then by 
$\det[d''(\omega)]=0$ and \eqref{d21},
we have 
\begin{equation} \label{dt}
(\pt_{\omega_0}\pt_{\omega_1}d(\omega))^2
=\pt_{\omega_0}^2d(\omega)\pt_{\omega_1}^2d(\omega)
=\omega_0(\pt_{\omega_0}^2d(\omega))^2.
\end{equation}
Let 
\[\xi
=(\xi_0,\xi_1)
=(-\omega_0\pt_{\omega_0}^2d(\omega),
\pt_{\omega_0}\pt_{\omega_1}d(\omega)).
\] 
Then $\xi$ is an eigenvector corresponding to the zero eigenvalue of $d''(\omega)$.
In order to calculate the value of 
$\frac{d^3}{d\lambda^3}d(\omega+\lambda\xi)|_{\lambda=0}$,
we establish the following lemma.

\begin{lemma} \label{sd3}
Let $\omega_1=2z_0\omega_0^{1/2}$.
\begin{enumerate}
\item
If $\pt_{\omega_0}\pt_{\omega_1}d(\omega)
=-\omega_0^{1/2}\pt_{\omega_0}^2d(\omega)$, then
\begin{align*}
&\frac{d^3}{d\lambda^3}d(\omega+\lambda\xi)|_{\lambda=0} \\
&=\omega_0^2(\pt_{\omega_0}^2d(\omega))^3[
-4\pt_{\omega_0}\pt_{\omega_1}^2d(\omega)
-4\omega_0^{1/2}\pt_{\omega_0}^2\pt_{\omega_1}d(\omega)
+\pt_{\omega_0}^2d(\omega)].
\end{align*}
\item
If $\pt_{\omega_0}\pt_{\omega_1}d(\omega)
=\omega_0^{1/2}\pt_{\omega_0}^2d(\omega)$, then
\begin{align*}
&\frac{d^3}{d\lambda^3}d(\omega+\lambda\xi)|_{\lambda=0} \\
&=\omega_0^2(\pt_{\omega_0}^2d(\omega))^3[
-4\pt_{\omega_0}\pt_{\omega_1}^2d(\omega)
+4\omega_0^{1/2}\pt_{\omega_0}^2\pt_{\omega_1}d(\omega)
+\pt_{\omega_0}^2d(\omega)].
\end{align*}
\end{enumerate}
\end{lemma}

\begin{proof}
By \eqref{d31}, \eqref{d32} and \eqref{dt}, we have
\begin{align*}
%%%%%%%%%%%%%%%
\xi_0^3\pt_{\omega_0}^3d(\omega)
&=-\omega_0^2(\pt_{\omega_0}^2d(\omega))^3\pt_{\omega_0}\pt_{\omega_1}^2d(\omega)
+\omega_0^2(\pt_{\omega_0}^2d(\omega))^4, \\
%%%%%%%%%%%%%%%
\xi_0^2\xi_1\pt_{\omega_0}^2\pt_{\omega_1}d(\omega)
&=\omega_0^2(\pt_{\omega_0}^2d(\omega))^2\pt_{\omega_0}\pt_{\omega_1}d(\omega)\pt_{\omega_0}^2\pt_{\omega_1}d(\omega), \\
%%%%%%%%%%%%%%%
\xi_0\xi_1^2\pt_{\omega_0}\pt_{\omega_1}^2d(\omega)
&=-\omega_0^2(\pt_{\omega_0}^2d(\omega))^4
-\omega_0^2(\pt_{\omega}^2d(\omega))^3\pt_{\omega_0}^3d(\omega), \\
%%%%%%%%%%%%%%%
\xi_1^3\pt_{\omega_1}^3d(\omega)
&=\omega_0^2(\pt_{\omega_0}^2d(\omega))^2\pt_{\omega_0}\pt_{\omega_1}d(\omega)\pt_{\omega_0}^2\pt_{\omega_1}d(\omega).
\end{align*}
These imply that
\begin{align*}
&\frac{d^3}{d\lambda^3}d(\omega+\lambda\xi)|_{\lambda=0} \\
&=\xi_0^3\pt_{\omega_0}^3d(\omega)
+3\xi_0^2\xi_1\pt_{\omega_0}^2\pt_{\omega_1}d(\omega)
+3\xi_0\xi_1^2\pt_{\omega_0}\pt_{\omega_1}^2d(\omega)
+\xi_1^3\pt_{\omega_1}^3d(\omega) \\
&=\omega_0^2(\pt_{\omega_0}^2d(\omega))^2[
-4\pt_{\omega_0}^2d(\omega)\pt_{\omega_0}\pt_{\omega_1}^2d(\omega) \\
&\qquad+4\pt_{\omega_0}\pt_{\omega_1}d(\omega)\pt_{\omega_0}^2\pt_{\omega_1}d(\omega)
+(\pt_{\omega_0}^2d(\omega))^2].
\end{align*}
Then we obtain the conclusion.
\end{proof}

\begin{proof}[Proof of Lemma \ref{d3}]
Let $\omega_1=2z_0\sqrt{\omega_0}$. 
Then, by \eqref{d21}, \eqref{d22}, \eqref{d33} and \eqref{d34},
we have
\begin{align*}
\pt_{\omega_0}^2d(\omega)
&=8\tilde\kappa_\omega\omega_0^{1/2}\alpha_{0,\omega}(z_0^2-\sigma+1)
+8\tilde\kappa_\omega\omega_0^{1/2}\alpha_{1,\omega}z_0(z_0^2-1), \\
\pt_{\omega_0}\pt_{\omega_1}d(\omega)
&=8\tilde\kappa_\omega\omega_0z_0(\sigma-2)\alpha_{0,\omega}
+8\tilde\kappa_\omega\omega_0(1-z_0^2)\alpha_{1,\omega}, \\
\pt_{\omega_0}\pt_{\omega_1}^2d(\omega)
&=\frac{4\omega_0^{1/2}\tilde\kappa_\omega\alpha_{0,\omega}}{\sigma(1-z_0^2)}[
(3\sigma-2)(\sigma-2)z_0^2
+(\sigma-1)^2(1-z_0^2)
] \\
&\quad+\frac{4(3\sigma-2)z_0\omega_0^{1/2}\tilde\kappa_\omega\alpha_{1,\omega}}{\sigma}, \\
\pt_{\omega_0}^2\pt_{\omega_1}d(\omega)
&=\frac{4z_0\tilde\kappa_\omega\alpha_{0,\omega}}{\sigma(1-z_0^2)}[
(2-3\sigma)(\sigma-2)-(\sigma-1)(1-z_0^2)
] \\
&\quad+\frac{4\tilde\kappa_\omega\alpha_{1,\omega}}{\sigma}(
-\sigma z_0^2+z_0^2-2\sigma+1
), 
\end{align*}
If 
$\pt_{\omega_0}\pt_{\omega_1}d(\omega)
=-\omega^{1/2}\pt_{\omega_0}^2d(\omega)$,
we have 
$-(1-z_0-\sigma)\alpha_{0,\omega}
=(1-z_0^2)\alpha_{1,\omega}$.
This implies that
\begin{align*}
&-4\pt_{\omega_0}\pt_{\omega_1}^2d(\omega)
-4\omega_0^{1/2}\pt_{\omega_0}^2\pt_{\omega_1}d(\omega)
+\pt_{\omega_0}^2d(\omega) \\
&=\frac{8\omega_0^{1/2}\tilde\kappa_\omega\alpha_{0,\omega}}
{\sigma(1-z_0^2)}\sigma(\sigma-1)(z_0-1)^2(z_0+1)
\ne0.
\end{align*}
Similarly, if 
$\pt_{\omega_0}\pt_{\omega_1}d(\omega)
=\omega^{1/2}\pt_{\omega_0}^2d(\omega)$,
we obtain
\begin{align*}
&-4\pt_{\omega_0}\pt_{\omega_1}^2d(\omega)
+4\omega_0^{1/2}\pt_{\omega_0}^2\pt_{\omega_1}d(\omega)
+\pt_{\omega_0}^2d(\omega) \\
&=\frac{8\omega_0^{1/2}\tilde\kappa_\omega\alpha_{0,\omega}}
{\sigma(1-z_0^2)}\sigma(\sigma-1)(z_0+1)^2(z_0-1)
\ne0.
\end{align*}
By Lemma \ref{sd3},
we conclude $\frac{d^3}{d\lambda^3}d(\omega+\lambda\xi)|_{\lambda=0}\ne0$.
\end{proof}

\section*{Acknowledgements}

The author would like to express his gratitude to Professor Masahito Ohta for helpful advices and encouragements.
The author would also like to thank Professor Masaya Maeda for giving useful notes to him.

\end{document}